\documentclass[11pt,a4paper]{article}
\usepackage[OT1]{fontenc}
\usepackage[english]{babel}
\usepackage{amsfonts}
\usepackage{amsthm}

\usepackage{graphicx}

\def\h{ {\cal H} }
\def\a{ {\cal A} }

\def\l{ {\cal L} }

\def\q{ {\cal Q} }
\def\b{ {\cal B} }

\def\t{ {\cal T} }
\def\s{ {\cal S} }
\def\e{ {\cal E} }
\def\o{ {\cal O} }
\def\p{ {\cal P} }

\def\k{ {\cal K} }

\newtheorem{teo}{Theorem}[section]
\newtheorem{prop}[teo]{Proposition}
\newtheorem{lem}[teo]{Lemma}
\newtheorem{coro}[teo]{Corollary}

\theoremstyle{definition}
\newtheorem{rem}[teo]{Remark}
\newtheorem{ejem}[teo]{Example}

\title{The matched projection and geodesics of the Grassmann manifold}
\author{Esteban Andruchow}

\begin{document}

\maketitle 

\begin{abstract}Given an idempotent operator $E$  in a complex Hilbert space $\h$, one can associate to it two orthogonal projections:
\begin{itemize}
\item The polar decomposition $2E-1=(2P-1)|2E-1|$ provides an orthogonal projection $P$. That the unitary part in the decomposition of $2E-1$ is of this form, i.e., a selfadjoint unitary operator,  is a remarkable observation done by G. Corach, H. Porta and L. Recht (see references below).
\item The question of which, among all orthogonal projections, is the one closest in norm to $E$, provides another projection, the so called {\it matched projection} $m(E)$, which answers this question. It was  found by X.  Tian, Q. Xu and  C. Fu (see references below).
\end{itemize}
In this paper we show that these projections coincide. Moreover, we show that there exists a unique minimal geodesic of the Grassmann manifold of $\h$ (the manifold of closed subspaces of $\h$) that joins $R(E)$ and $R(E^*)$. The orthogonal projection onto the midpoint of this geodesic, also coincides with $m(E)$.
 \end{abstract}

\bigskip



\section{Introduction}

Let $E$ be an idempotent in $\b(\h)$, which written as a $2\times 2$ matrix in terms of the decomposition $\h=R(E)\oplus R(E)^\perp$ is of the form
\begin{equation}\label{matriz E}
E=\left(\begin{array}{cc} 1 & B \\ 0 & 0\end{array} \right),
\end{equation}
where $B:R(E)^\perp\to R(E)$. If $\s\subset \h$ is a closed subspace, denote by $P_\s$ the orthogonal projection onto $\s$. In this note we study the position of the projections $P_{R(E)}$ and $P_{R(E^*)}=P_{N(E)}^\perp$ as points in the Grassmann manifold $\p(\h)$, of all orthogonal projections in $\h$.  The points $P_{R(E)}$ and $P_{R(E^*)}$   lie at distance strictly less than $1$, namely (Theorem \ref{teo 24})
$$
\|P_{R(E)}-P_{R(E^*)}\|= \frac{\|B\|}{(1+\|B\|^2)^{1/2}}.
$$
In \cite{pr} H. Porta and L. Recht, introduced a natural linear connection in $\p(\h)$, and proved that two projections $P,Q\in\p(\h)$ such that  $\|P-Q \|<1$  are joined by a unique geodesic of $\p(\h)$  (up to reparametrization), which is minimal for the Finsler metric given by the usual norm of $\b(\h)$ at very tangent space.  It follows that $P_{R(E)}$ and $P_{R(E^*)}$ are joined at times $t=0$ and $t=1$ by a (unique, minimal) geodesic $\delta_E$, which is of the form $\delta_E(t)=e^{itX_E}P_{R(E)}e^{-itX_E}$, for $X_E^*=X_E$, $\|X_E\|<\pi/2$, co-diagonal in the above $2\times 2$ matrix representation.

In \cite{TianXuFu} X. Tian, Q. Xu and  C.  Fu introduced the {\it matched projection} $m(E)$ of an idempotent operator $E$. Among  remarkable properties, they prove that $m(E)$ is the orthogonal projection which is closest in norm to $E$:
$$
\|m(E)-E\|\le \|P-E\|, \ \hbox{ for all } P\in\p(\h).
$$
In \cite{cpr}, G. Corach, H. Porta and L. Recht studied the inmersion $\p(\h)\hookrightarrow \q(\h)$,  of $\p(\h)$ into the larger space $\q(\h)=\{E\in\b(\h): E^2=E\}$ of idempotent operators, or oblique projections. For instance, they prove that the $C^\infty$-map
$$
\pi:\q(\h)\to\p(\h), 
$$
is a retraction, where $2E-1=(2\pi(E)-1)|2E-1|$ is the polar decomposition of $2E-1$.

In this paper we show that
$$
m(E)=\delta_E(\frac12)=\pi(E).
$$
The contents of the paper are the following. In Section 2 we give basic properties of $R(E)$ and $R(E^*)$. In Section 3 we compute the spectrum (and norm) of $P_{R(E)}-P_{R(E^*)}$. In Section 4 we characterize the geodesic $\delta_E$, by giving an expression of the exponent $X_E$ in terms of the operator $B$ of (\ref{matriz E}). In Section 5 we consider the matched projection $m(E)$ of $E$ \cite{TianXuFu}, and show that $m(E)=\delta_E(\frac12)$.  In Section 6 we use Halmos' theory for two subspaces \cite{halmos} and Davis' results on the difference of orthogonal projections \cite{davis}, to give further descriptions of $m(E)$. In Section 7 we study an example: $E_a=\frac12(\Gamma_a-1)$, where $\Gamma_a$ is the composition operator $\Gamma_af=f\circ\varphi_a$ acting in $\h=L^2(\mathbb{T}, dm)$ ($\mathbb{T}$ the unit circle, $dm$ the normalized Lebesgue measure), and $\varphi_a(z)=\frac{a-z}{1-\bar{a}z}$, for $|a|<1$. In Section 8 we consider the retraction $\pi:\q(\h)\to\p(\h)$ and prove that $\pi(E)=\delta_E(\frac12)$.  We also examine certain metric properties of $m(E)$.

\section{Preliminaries}

We say that an operator $C$ is a {\it reflection} if $C^2=1$; a {\it symmetry} $S$ is a selfadjoint reflection.
We identify a closed subspace $\s\subset\h$ with the orthogonal projection $P_\s$ onto $\s$, and with the symmetry $2P_\s-1$ which equals the identity on $\s$ and minus the identity on $\s^\perp$. The set of subspaces, projections or symmetries will be denoted $\p(\h)$. Similarly, we denote by $\q(\h)$ the space of idempotents $Q$ or  reflections $2Q-1$.

When considering a pair of subspaces $\s,\t\subset\h$ (or projections $P_\s.P_\t$) as points in $\p(\h)$, a  key geometric fact  are the intersections $\s\cap\t^\perp$ and $\s^\perp\cap\t$: the necessary and sufficient condition for the existence geodesics of $\p(\h)$ joining $P_\s$ and $P_\t$ is that these  intersections have the same dimension; the geodesic is unique (up to reparametrization) iff these intersections are trivial. The geodesic $\delta$ which joins $\delta(0)=\s$ and $\delta(1)=\t$ ($\s$ and $\t$ satisfying the above condiion) takes the form of a one-parameter group of unitaries acting on $\s$. The infinitesimal generator of this group is a bounded selfadjoint operator $X$, which can be chosen with norm $\|X\|\le\pi/2$, and its $2\times 2$ matrix with respect to $\s\oplus\s^\perp=\h$ (and also with respecto to $\t\oplus\t^\perp=\h$) is codiagonal. The geodesic is then, as a curve of  subspaces
$$
\delta(t)=e^{itX}\s,
$$
or 
$$
\delta(t)=e^{itX}P_\s  e^{-itX} , \ \delta(t)=e^{itX}(2P_\s-1)  e^{-itX}
$$
as projections or symmetries. Regarded as symmetries, it has the useful computational feature, that since $X$ is codiagonal with respect to  $\s\oplus\s^\perp=\h$, $X$ anti-commutes with $2P_\s-1$. Therefore $\delta$ takes the alternative forms 
$$
\delta(t)=e^{2itX}(2P_\s-1)=(2P_\s-1)  e^{-2itX}.
$$
These facts were shown in \cite{pr}.

In our case,
$$
R(E)\cap R(E^*)^\perp=R(E)\cap N(E)=\{0\} \hbox{ and } R(E)^\perp\cap R(E^*)=N(E^*)\cap R(E^*)=\{0\},
$$
so that there exists a unique geodesic joining $P_{R(E)}$ and $P_{R(E^*)}$.

Also note that
\begin{equation}\label{B rango denso?}
R(E)\cap R(E^*)=\{f\in R(E): B^*f=0\}=N(B^*),
\end{equation}
and 
\begin{equation}\label{B nucleo trivial?}
R(E)^\perp\cap R(E^*)^\perp=\{g\in R(E)^\perp: Bg=0\}=N(B).
\end{equation}
Two subspaces $\s, \t$ are in {\it generic position} \cite{halmos} if $\s\cap\t=\s\cap\t^\perp=\s^\perp\cap\t=\s^\perp\cap\t^\perp=\{0\}$.
In general, we call {\it generic part} of $\s$ and $\t$ the subspace
$$
\h_0=\left(\s\cap\t \oplus \s\cap\t^\perp \oplus \s^\perp\cap\t \oplus \s^\perp\cap\t^\perp\right)^\perp.
$$
In our case $R(E)$ and $R(E^*)$ are in generic position  iff $B$ has trivial nullspace and dense range. The generic part of $R(E)$ and $R(E^*)$ 
$$
\h_0=\left(R(E)\cap R(E^*)\oplus R(E)^\perp\cap R(E^*)^\perp\right)^\perp=\left(N(B)\oplus N(B^*)\right)^\perp
$$
is trivial iff $B=0$, i.e. $E=E^*$ is an orthogonal projection. Since $E$ and $E^*$ coincide in $\h_0^\perp$, in order to characterize the unique geodesic between $P_{R(E)}$ and $P_{R(E^*)}$, we must focus on the subspace $\h_0$. 

\section{The distance between $R(E)$ and $R(E^*)$}

Let us first recall a result by D. Buckholtz \cite{buckholtz}:
\begin{rem} \cite{buckholtz} Let $\s,\t\subset\h$ be closed subspaces. Then the following are equivalent:
\begin{enumerate}
\item
$\s\dot{+}\t=\h$.
\item
$\|P_\s+P_\t-1\|<1$.
\item
$P_\s-P_\t$ is invertible.
\end{enumerate}
\end{rem}
As an immediate consequence we get
\begin{coro}
$\|P_{R(E)}-P_{R(E)}\|<1$.
\end{coro}
\begin{proof}
Put $\s=R(E)$ and $\t=N(E)=R(E^*)^\perp$. Then the first condition is verified, and therefore the second says that
$$
\|P_{R(E)}+P_{N(E)}-1\|=\|P_{R(E)}-P_{N(E)^\perp}\|=\|P_{R(E)}-P_{R(E^*)}\|<1.
$$ 
\end{proof}
Let us now estimate this distance in terms of the operator $B=B_E:R(E)^\perp\to R(E)$ of (\ref{matriz E}).
To this effect, it will be useful to have formulas of $P_{R(E)}$ and $P_{R(E^*)}$. In \cite{ando}, it was shown that (note that $E+E^*-1$ is invertible)
\begin{equation}\label{ando1}
P_{R(E)}=E(E+E^*-1)^{-1},
\end{equation}
and 
\begin{equation}\label{ando2}
P_{R(E^*)}=E^*(E^*+E-1)^{-1}.
\end{equation}
Therefore
$$
P_{R(E)}-P_{R(E^*)}=(E-E^*)(E+E^*-1)^{-1}.
$$
Let us write this expression in terms of $B$ (see (\ref{matriz E})). To compute $(E+E^*-1)^{-1}$, note that
$$
E+E^*-1=\left(\begin{array}{cc} 1 & B \\ B^* & -1 \end{array} \right)
$$
and 
$$
(E+E^*-1)^2=\left(\begin{array}{cc} 1+BB^* & 0 \\ 0 & 1+B^*B \end{array} \right).
$$
Then 
$$
(E+E^*-1)^{-1}=\left(\begin{array}{cc} (1+BB^*)^{-1} & 0 \\ 0 & (1+B^*B)^{-1} \end{array} \right)
\left(\begin{array}{cc} 1 & B \\ B^* & -1 \end{array} \right)
$$
$$
=\left(\begin{array}{cc} (1+BB^*)^{-1} & (1+BB^*)^{-1}B \\ (1+B^*B)^{-1}B^* & -(1+B^*B)^{-1} \end{array} \right).
$$
Therefore
\begin{equation}\label{diferencia}
P_{R(E)}-P_{R(E^*)}=\left(\begin{array}{cc} B(1+B^*B)^{-1}B^* & -B(1+B^*B)^{-1} \\ -B^*(1+BB^*)^{-1} & -B^*(1+BB^*)^{-1}B \end{array} \right).
\end{equation}
Since  $B^*(BB^*)^k=(B^*B)^kB^*$ for any $k\ge 0$, we get that 
$$
B^*(1+BB^*)^{-1}B=(1+B^*B)^{-1}B^*B=B^*B(1+B^*B)^{-1}.
$$
 Similarly, 
$$
B(1+B^*B)^{-1}B^*=(1+BB^*)^{-1}BB^*=BB^*(1+BB^*)^{-1}.
$$
Then we have also 
\begin{equation}\label{diferencia2}
P_{R(E)}-P_{R(E^*)}=\left(\begin{array}{cc} BB^*(1+BB^*)^{-1} & -B(1+B^*B)^{-1} \\ -B^*(1+BB^*)^{-1} & -B^*B(1+BB^*)^{-1} \end{array} \right).
\end{equation}
We shall need the following result by G. Corach, H. Porta and L. Recht \cite{cpr} (Corollary 1.7):
\begin{rem}\label{corolario cpr}\cite{cpr} Put $\h_1=R(E)\cap N(E)^\perp$, $\h_2=R(E)^\perp\cap N(E)$, $\h_3=R(E)\ominus \h_1$ and $\h_4=R(E)^\perp\ominus\h_2$. Then, in the decomposition $\h=\h_1\oplus\h_2\oplus\h_3\oplus\h_4$, $E$ is written
$$
E=1\oplus 0 \oplus\left(\begin{array}{cc} 1 & -Q^{1/2}(1-Q)^{-1/2}D \\ 0 & 0 \end{array}\right).
$$
Here $Q$ is a positive contraction in $\h_3$ with $N(Q)=0$ and $1-Q$ invertible, and $D:\h_4\to\h_3$ is an isometric isomorphism.
\end{rem}
Denote by $R$ the selfadjoint operator
\begin{equation}\label{R o Q?}
R:=-Q^{1/2}(1-Q)^{-1/2}
\end{equation}.
Then, with these notations,  we have
\begin{teo}\label{teo 24}
Suppose that $R(E)$ and $N(E)$ are infinite dimensional.
Then 
$$
\sigma(P_{R(E)}-P_{R(E^*)})\setminus\{0\}=\{\pm \frac{t}{\sqrt{1+t^2}}: t\in\sigma(R), t\ne 0\}.
$$  

In particular,
$$
\|P_{R(E)}-P_{R(E^*)}\|= \frac{\|B\|}{(1+\|B\|^2)^{1/2}}.
$$
\end{teo}
\begin{proof}
{\bf Case 1.} Suppose first that $\h=\l\times \l$ and $B:\l\to\l$ satisfies $B^*=B$. Then, using (\ref{diferencia2}) we have
$$
P_{R(E)}-P_{R(E^*)}=\left( \begin{array}{cc} B^2(1+B^2)^{-1} & -B(1+B^2)^{-1} \\ -B(1+B^2)^{-1} & -B^2(1+B^2)^{-1} \end{array}\right).
$$
We can factorize  $P_{R(E)}-P_{R(E^*)}=ST$, where
$$
S=\left( \begin{array}{cc} B(1+B^2)^{-1/2} & -(1+B^2)^{-1/2} \\ -(1+B^2)^{-1/2} & -B(1+B^2)^{-1/2} \end{array}\right) \hbox{ and } T=\left( \begin{array}{cc} B(1+B^2)^{-1/2} & 0 \\ 0 & B(1+B^2)^{-1/2} \end{array}\right).
$$
The operator $S$ is a symmetry: $S^*=S$ and $S^2=1$. Also it is clear that $S$ and $T$ commute. Consider $\a$ the (commutative) C$^*$-algebra generated by $S$ and $T$. The spectrum $\sigma(P_{R(E)}-P_{R(E^*)})$ can be computed
$$
\sigma(P_{R(E)}-P_{R(E^*)})=\{\varphi(S)\varphi(T): \varphi \hbox{ is a character in } \a\}.
$$
Since $S$ is a symmetry, $\varphi(S)=\pm 1$. On the other hand $\varphi(T)$ takes (all possible values) in the spectrum of $T$, i.e., 
$\sigma(T)=\sigma(B(1+B^2)^{-1/2})=\{ \frac{t}{\sqrt{1+t^2}}: t\in\sigma(B)\}$. Since $P_{R(E)}-P_{R(E^*)}$ is a difference of orthogonal projections, its spectrum is symmetric with respect to the origin (see for instance \cite{davis}): save for $\pm 1$, which may or may not belong to this spectrum - and in this case they do not, because the norm of the difference is strictly less than $1$, we have that
$\lambda\in\sigma(P_{R(E)}-P_{R(E^*)})$ iff $-\lambda\in\sigma(P_{R(E)}-P_{R(E^*)})$. Therefore we have that in this case
$$
\sigma(P_{R(E)}-P_{R(E^*)})=\{\pm\frac{t}{\sqrt{1+t^2}}: t\in\sigma(B)\}.
$$ 
In particular, since the continuous function $f(t)=t(1+t^2)^{-1/2}$ is increasing for all $t\in\mathbb{R}$, and $B$ is selfadjoint (which implies that either $+\|B\|$ or $-\|B\|$ belong to $\sigma(B)$)  we have that 
$$
\|B(1+B^2)^{-1/2}\|=\sup\{|f(t)|: t\in\sigma(B)\}=\max\{-f(-\|B\|), f(\|B\|)\}=f(\|B\|)
$$
$$
=\frac{\|B\|}{(1+\|B\|^2)^{1/2}}.
$$
It follows that  $\|P_{R(E)}-P_{R(E^*)}\|=\frac{\|B\|}{(1+\|B\|^2)^{1/2}}$ in this case.

{\bf Case 2.} (General case): by the result of Corach, Porta and Recht (Corollary 1.7 in \cite{cpr}) recalled in Remark \ref{corolario cpr}, we have that in certain orthogonal decomposition $\h=\h_1\oplus\h_2\oplus\h_3\oplus\h_4$, $E$ can be written
$$
E=1\oplus 0 \oplus \left(\begin{array}{cc} 1 & RD \\ 0 & 0 \end{array} \right),
$$
where $R=-Q^{1/2}(1-Q)^{-1/2}$ is selfadjoint in $\h_3$ (as in Remark \ref{corolario cpr}), and $D:\h_4\to\h_3$ is a unitary transformation. Denote by $E'$ the idempotent $\left(\begin{array}{cc} 1 & RD \\ 0 & 0 \end{array} \right)$ in $\h_3\oplus \h_4$. Clearly, it suffices to establich our claim for $E'$ (on $\h_1\oplus\h_2$, $E$ is selfadjoint, and therefore coincides with $E^*$, $P_{R(E)}$ and $P_{R(E^*)}$). Consider the unitary operator ${\bf D}:\h_3\oplus \h_4\to \h_3\times\h_3$ given by
$$
{\bf D}(f+g)=(f, D^*g).
$$
Then ${\bf D}E'{\bf D}^*$ is a idempotent in $\h_3\times\h_3$, given by
$$
{\bf D}E'{\bf D}^*(f,g)={\bf D}E'(f+D^*g)={\bf D}\left(\begin{array}{cc} 1 & RD \\ 0 & 0 \end{array}\right)\left( \begin{array}{c} f \\ D^*g\end{array}\right) \begin{array}{c} \h_3 \\ \h_4 \end{array}= {\bf D} \left( \begin{array}{c} f+RDD^*g \\ 0 \end{array}\right) \begin{array}{c} \h_3 \\ \h_4 \end{array}
$$
$$
=(f+Rg,0),
$$
i.e., ${\bf D}E'{\bf D}^*=\left(\begin{array}{cc} 1 & R \\ 0 & 0\end{array} \right)$ in $\h_3\times\h_3$.  Since $({\bf D}E'{\bf D}^*)^*={\bf D}(E')^*{\bf D}^*$ and
$$
P_{R({\bf D}E'{\bf D}^*)}-P_{R({\bf D}(E')^*{\bf D}^*)}={\bf D}\left(P_{R(E')}-P_{R((E')^*)}\right){\bf D}^*,
$$
we have, by the previous case, 
$$
\sigma(P_{R(E')}-P_{R((E')^*})=\{\pm \frac{t}{\sqrt{1+t^2}}: t\in\sigma(R)\}, 
$$ 
and thus 
$$
\sigma(P_{R(E)}-P_{R((E)^*})\setminus\{0\}=\{\pm \frac{t}{\sqrt{1+t^2}}: t\in\sigma(R), t\ne 0\}.
$$ 
And also
$$
\|P_{R(E)}-P_{R((E)^*)}\|=\|P_{R(E')}-P_{R((E')^*)}\|=\frac{\|R\|}{(1+\|R\|^2)^{1/2}}=\frac{\|B\|}{(1+\|B\|^2)^{1/2}},
$$
because $\|B\|=\|RD\|=\|R\|$.
\end{proof}

\begin{rem}
We can write the spectrum $P_{R(E)}-P_{R((E)^*}$ in terms of $Q$ (instead of $R$).
Note that
$$
R(1+R^2)^{-1/2}=-Q^{1/2}.
$$
Thus we get 
$$
\sigma(P_{R(E)}-P_{R(E)^*})\setminus\{0\}=\{s: s^2\in\sigma(Q), s \ne 0\}.
$$ 
\end{rem}
\begin{rem}
The value $0$ may or may not lie in the spectrum of $P_{R(E)}-P_{R(E^*)}$. Indeed, since
$$
P_{R(E)}-P_{R(E^*)}=(E-E^*)(E+E^*-1)^{-1},
$$
it follows that $P_{R(E)}-P_{R(E^*)}$ is invertible if and only if $E-E^*$ is invertible. Note that 
$$
E-E^*=\left(\begin{array}{cc} 0 & B \\ -B^* & 0 \end{array}\right)
$$
is invertible if and only if $B:R(E)^\perp\to R(E)$ is an isomorphism.
\end{rem}

The geodesic distance $d_g(P,Q)$ between two projections $P$ and $Q$ is defined as the infinum of the lengths of all piecewise smooth curves joining $P$ and $Q$ inside $\p(\h)$, where the length $\ell(\gamma)$ of a piecewise smooth curve $\gamma$ parametrized in the interval $I$ is measured by
$$
\ell(\gamma)=\int_I \|\dot{\gamma}(t)\|dt,
$$
using the usual norm of $\b(\h)$. It has been shown that (see for instance \cite{paisanos}), that
$d_g(P,Q)=\arcsin\|P-Q\|$.
\begin{coro}
Let $E=\left(\begin{array}{cc} 1 & B \\ 0 & 0\end{array} \right)$, with $R(E)$ and $N(E)$ infinite dimensional. Then
$$
d_g(P_{R(E)},P_{R(E^*)})=\arctan \|B\|.
$$
\end{coro}
\begin{proof}
Denote $d=d_g(P_{R(E)},P_{R(E^*)})$. Note that $0\le d\le \pi/2$. We know that $\|P_{R(E)}-P_{R(E^*)}\|=\frac{\|B\|}{(1+\|B\|^2)^{1/2}}=\sin(d)$. Then $\cos(d)=(1-\sin(d)^2)^{1/2}=\frac{1}{(1+\|B\|)^{1/2}}$. Thus $\tan(d)=\|B\|$.
\end{proof}

\section{The geodesic between $R(E)$ and $R(E^*)$}

Since $\|P_{R(E)}-P_{R(E^*)}\|<1$, there exists a unique operator $X_E$, the velocity vector of the geodesic $\delta_E$ at $t=0$, such that $\delta_E(0)=P_{R(E)}$ and $\delta_E(1)=P_{R(E^*)}$. The operator $X_E$ thus satisfies
$$
X_E^*=X_E, \ \|X_E\|<\pi/2, \ X_E \hbox{ is  } R(E)\oplus R(E)^\perp \hbox{ co-diagonal},
$$
(meaning that its matrix with respect to this decomposition is co-diagonal, or equivalently, $X_E(R(E))\subset R(E)^\perp$ and $X_E(R(E)^\perp)\subset R(E)$) and
$$
e^{iX_E}R(E)=R(E^*).
$$
The geodesic si of the form $\delta_E(t)=e^{itX_E}P_{R(E)}e^{-itX_E}$ (and this last condition is \\ $e^{iX_E}P_{R(E)}e^{-iX_E}=P_{R(E^*)}$ (see \cite{cpr})).

For certain computations, it is useful to consider symmetries instead of projections:  to the closed subspace $\s\subset\h$ the symmetry $2P_\s-1$, which equals the identity in $\s$ and minus the indentity in $\s^\perp$.  This is the standpoint in \cite{pr}. The fact that $X_E$ is co-diagonal tranlates to the fact that $X_E$ anti-commutes with $2P_{R(E)}-1$. It follows that $e^{itX_e}(2P_{R(E)}-1)=(2P_{R(E)}-1) e^{-itX_E}$. Thus the geodesic $\epsilon(t)=2\delta(t)-1$ is given by (see\cite{pr})
$$
\epsilon(t)=e^{2itX_e}(2P_{R(E)}-1)=(2P_{R(E)}-1)e^{-2itX_E}.
$$
Let us compute the exponent $X_E$ associated to $E$ (with $R(E)$ and $N(E)$ infinite dimensional). First we analize the case when $\h=\l\times\l$ and $B:\l\to\l$ is selfadjoint.
\begin{lem}\label{lema B}
If $\h=\l\times\l$ and $E=\left(\begin{array}{cc} 1 & B \\ 0 & 0 \end{array}\right)$, with $B:\l\to\l$ selfadjoint, we have that
$$
X_E=\left(\begin{array}{cc} 0 & i \arctan(B)  \\  -i \arctan(B) & 0 \end{array} \right).
$$
\end{lem}
\begin{proof}
In view of the above comment, we have that 
$e^{2iX_e}(2P_{R(E)}-1)=2P_{R(E^*)}-1$,  i.e.,
$$
e^{2iX_E}=(2P_{R(E^*)}-1)(2P_{R(E)}-1)=\left(2E^*(E+E^*-1)^{-1}-1\right)(2P_{R(E)}-1)
$$
$$
=(E^*-E+1)(E+E^*-1)^{-1}(2P_{R(E)}-1)
=\left(\begin{array}{cc} 1 & -B \\ B & 1 \end{array}\right)\left(\begin{array}{cc} 1 & B \\ B & -1 \end{array}\right)^{-1}\left(\begin{array}{cc} 1 & 0 \\ 0 & -1 \end{array}\right)
$$
$$
=\left(\begin{array}{cc} 1 & -B \\ B & 1 \end{array}\right)\left(\begin{array}{cc} 1 & B \\ B & -1 \end{array}\right)\left(\begin{array}{cc} (1+B^2)^{-1} & 0 \\ 0 & (1+B^2)^{-1} \end{array}\right)\left(\begin{array}{cc} 1 & 0 \\ 0 & -1 \end{array}\right),
$$
using that $\left(\begin{array}{cc} 1 & B \\ B & -1 \end{array}\right)^{-1}=\left(\begin{array}{cc} 1 & B \\ B & -1 \end{array}\right)\left(\begin{array}{cc} (1+B^2)^{-1} & 0 \\ 0 & (1+B^2)^{-1} \end{array}\right)$. Since the right hand two matrices in the product above commute, we have
$$
e^{2iX_E}=\left(\begin{array}{cc} 1 & -B \\ B & 1 \end{array}\right)\left(\begin{array}{cc} 1 & B \\ B & -1 \end{array}\right)\left(\begin{array}{cc} 1 & 0 \\ 0 & -1 \end{array}\right)\left(\begin{array}{cc} (1+B^2)^{-1} & 0 \\ 0 & (1+B^2)^{-1} \end{array}\right)
$$
$$
=\left(\begin{array}{cc} (1-B^2)(1+B^2)^{-1} & -2B(1+B^2)^{-1} \\ 2B(1+B^2)^{-1} & (1-B^2)(1+B^2)^{-1} \end{array}\right).
$$
Note that $C=:(1-B^2)(1+B^2)^{-1}$ and $S:=2B(1+B^2)^{-1}$ are selfadjoint operators, with $-1+d\le C\le 1-d$, $-1+d\le S\le 1-d$ for some $0<d<1$, and $C^2+S^2=1$. Then there exists $Z^*=Z$ in $\b(\l)$, $\|Z\|<\pi/2$ such that $C=\cos(Z)$ and $S=\sin(Z)$. Thus
$$
e^{2iX_E}=\left(\begin{array}{cc} \cos(Z) & -\sin(Z) \\ \sin(Z) & \cos(Z)\end{array}\right), \hbox{ i.e. } , X_E=\left(\begin{array}{cc} 0 & \frac{i}{2}Z \\ -\frac{i}{2}Z & 0 \end{array}\right).
$$
Finally, note that $\frac12 Z=\frac12\arcsin(2B(1+B^2))=\arctan(B)$, by the functional equality 
$$
\frac12 \arcsin(\frac{2t}{1+t^2})=\arctan(t)
$$
valid for all $t\in\mathbb{R}$.

\end{proof}

We analize now the general case. Again we invoke the result by Corach, Porta and Recht \cite {cpr}  (Remark \ref{corolario cpr} above):  in the orthogonal decomposition $\h=\h_1\oplus\h_2\oplus\h_3\oplus\h_4$, $E=1\oplus 0 \oplus \left(\begin{array}{cc} 1 & RD \\ 0 & 0 \end{array} \right)$,
where $R$ is selfadjoint in $\h_3$  and $D:\h_4\to\h_3$ is a unitary transformation. 
Then, using this decomposition, 
$$
E^*=1 \oplus 0 \oplus \left(\begin{array}{cc} 1 &  0 \\ D^*R & 0 \end{array} \right)
$$
and thus
$$
e^{2iX_E}=(E^*-E+1)(E+E^*-1)^{-1}(2P_{R(E)}-1)
$$
$$
=1\oplus 1\oplus ((E')^*-E'+1)(E'+(E')^*-1)^{-1}(2P_{R(E')}-1),
$$
where again 
$$
E'=\left(\begin{array}{cc} 1 & RD \\ 0 & 0 \end{array} \right)={\bf D}^*\left(\begin{array}{cc} 1 & R \\ 0 & 0 \end{array} \right){\bf D}
$$
for the unitary ${\bf D}:\h_3\oplus \h_4\to \h_3\times\h_3$, ${\bf D}(f+g)=(f, D^*g)$.
Then 
$
((E')^*-E'+1)(E'+(E')^*-1)^{-1}(2P_{R(E')}-1)$
equals
$$
{\bf D}^*\left(\begin{array}{cc} (1-R^2)(1+R^2)^{-1} & -2R(1+R^2)^{-1} \\ 2R(1+R^2)^{-1} & (1-R^2)(1+R^2)^{-1} \end{array}\right){\bf D}.
$$
Therefore we have
\begin{coro}\label{coro 32}
Let $E$ be an idempotent with $R(E)$ and $N(E)$ infinite dimensional. Then, with the current notations, the exponent $X_E$ of the unique geodesic in $\p(\h)$ joining $P_{R(E)}$ a $t=0$ and $P_{R(E^*)}$ at $t=1$ is
$$
X_E= 0 \oplus 0 \oplus \left(\begin{array}{cc} 0 & i\arctan(R)D \\ -iD^*\arctan(R) & 0 \end{array}\right).
$$  
\end{coro}
\begin{proof}
In $\h_1\oplus\h_2$, $e^{2iX_E}$ equals the identity, so $X_E=0$ there. The rest of $X_E$ is deduced from Lemma \ref{lema B}:
$$
{\bf D}^*\left(\begin{array}{cc} 0 & i\arctan(R) \\ -i\arctan(R) & 0 \end{array}\right) {\bf D}=\left(\begin{array}{cc} 0 & i\arctan(R)D \\ -iD^*\arctan(R) & 0 \end{array}\right).
$$
\end{proof}

\begin{rem}\label{punto medio}
We can therefore compute de midpoint $\gamma(\frac12)$ of the geodesic joining $P_{R(E)}$ and $P_{R(E^*)}$. We choose to write the corresponding symmetry $2\gamma(\frac12)-1=e^{iX_E}(2P_{R(E)}-1)$.
Note that in the decomposition $\h=\bigoplus_{i=1}^4 \h_i$ given in Remark \ref{corolario cpr}, we have that 
$$
e^{iX_E}=1\oplus 1 \oplus e^{\left( \begin{array}{cc} 0 & -\arctan(R) D \\ D^* \arctan(R) & 0 \end{array} \right)},
$$
which after elementary computations, equals
$$
1\oplus 1 \oplus \left( \begin{array}{cc} \cos\left(\arctan(R)\right) & -\sin\left(\arctan(R)\right) D \\  D^* \sin\left(\arctan(R)\right) & D^* \cos\left(\arctan(R)\right) D \end{array} \right).
$$
Since $\cos(\arctan(t))=\frac{1}{\sqrt{1+t^2}}$ and $\sin(\arctan(t))=\frac{t}{\sqrt{1+t^2}}$ we get that this expression above equals
$$
1\oplus 1 \oplus \left( \begin{array}{cc} (1+R^2)^{-1/2} & -R(1+R^2)^{-1/2}  D \\  D^* R(1+R^2)^{-1/2}  & D^* (1+R^2)^{-1/2} D \end{array} \right).
$$
Therefore we get 
\begin{equation}\label{midpoint}
2\gamma(\frac12)-1=
1\oplus 1 \oplus \left( \begin{array}{cc} (1+R^2)^{-1/2} & R(1+R^2)^{-1/2}  D \\  D^* R(1+R^2)^{-1/2}  & -D^* (1+R^2)^{-1/2} D \end{array} \right).
\end{equation}
In the special case when $\h=\l\times\l$ and $B:\l\to\l$ is selfadjoint, we have
\begin{equation}\label{midpointbis}
2\gamma(\frac12)-1=\left( \begin{array}{cc} (1+B^2)^{-1/2} & B(1+B^2)^{-1/2}  \\  B(1+B^2)^{-1/2}  & - (1+B^2)^{-1/2}  \end{array} \right).
\end{equation}
\end{rem}

\section{The matched projection}

In \cite{TianXuFu}, X. Tian, Q. Xu and C. Fu defined the {\it matched projection}  $m(E)$ associated to an idempotent $E$. It has remarkable properties with respect to the operator order (see also \cite{ZhaoFangLi}). For $E=\left(\begin{array}{cc} 1 & B \\ 0 & 0\end{array}\right)$  in the decompostion $\h=R(E)\oplus R(E)^\perp$, $m(E)$  is defined as
\begin{equation}\label{matched}
m(E)=\frac12\left( \begin{array}{cc} (T+1)T^{-1} & T^{-1}B \\ B^*T^{-1} & B^*\left( T(T+1)\right)^{-1}B \end{array} \right) ,  \ \hbox{ for } T:=(1+BB^*)^{1/2}.
\end{equation} 
\begin{rem}\label{props matched}
Let us state some of the properties of $m(E)$ (see \cite{TianXuFu}).
\begin{enumerate}
\item
For any orthogonal projection $P$,
$$
\|m(E)-E\|\le \|P-E\|\le \|1-m(E)-E\|.
$$
\item
$\|m(E)-E\|=\frac12\{\|E\|-1+\sqrt{\|E\|^2-1}\}$.
\item
$m(E)=m(E^*)$, $m(1-E)=1-m(E)$.
\item
If $\k$ is a Hilbert space and $\Phi:\b(\h)\to\b(\k)$ is a $*$-homomorphism, then
$m(\Phi(E))=\Phi(m(E))$.
\end{enumerate}
\end{rem}
The matched projection has the following geometric interpretation:
\begin{teo}
Let $E$ be an idempotent, and $\gamma_E$ the unique geodesic of $\p(\h)$ such that $\gamma_E(0)=P_{R(E)}$ and $\gamma_E(1)=P_{R(E^*)}$. Then
$$
m(E)=\gamma_E(\frac12).
$$
\end{teo}
\begin{proof}
As before, we consider first the case  $\h=\l\times\l$ and $B:\l\to\l$ selfadjoint.
Then the symmetry $2m(E)-1$ associated to $m(E)$ is
$$
\left( \begin{array}{cc} \left((1+B^2)^{1/2}+1\right)(1+B^2)^{-1} -1 & (1+B^2)^{-1}B \\ B(1+B^2)^{-1} & B \left( (1+B^2)^{1/2}\left((1+B^2)^{1/2}+1\right)\right)^{-1} B-1 \end{array}\right)
$$
which after elementary computations equals
$$
\left( \begin{array}{cc} (1+B^2)^{-1/2} & B(1+B^2)^{-1/2}  \\ B(1+B^2)^{-1/2} & -(1+B^2)^{-1/2}\end{array}\right)
$$
which is precisely the expression of $2\delta_E(\frac12)-1$ given in (\ref{midpointbis})  in Remark \ref{punto medio}. 

Consider now the general case, and again the decomposition $\h=\bigoplus_{i=1}^4\h_i$ in Remark \ref{corolario cpr}. In $\h_1\oplus\h_2$, $E$ is an orthogonal projection, and thus $m(E)=E=P_{R(E)}=P_{R(E^*)}$. Let us analize $\h_3\oplus\h_4$, and denote  by $E'$ the reduction of $E$ there. Here $E'=\left(\begin{array}{cc} 1 & RD \\ 0 & 0 \end{array}\right)$, and according to (\ref{matched})  (see \cite{TianXuFu}): for $T:=(1+BB^*)^{1/2}$, after elementary computations
$$
2 m(E')-1=\left( \begin{array}{cc} (T+1)T^{-1}-1 & T^{-1}B \\ B^*T^{-1} & B^*\left( T(T+1)\right)^{-1}B -1\end{array} \right)
$$
$$
=\left( \begin{array}{ccc} (1+BB^*)^{-1/2} & & (1+BB^*)^{-1/2} \\  B^*(1+BB^*)^{-1/2} & & B^*\left( (1+BB^*)^{1/2}\left((1+BB^*)^{1/2}+1\right)\right)^{-1}-1 \end{array}\right).
$$
Note that for any $k\ge 0$, $(BB^*)^kB=B(B^*B)^k$, and thus for any continuous function $g$ in the spectrum of $BB^*$ and $B^*B$, one has $g(BB^*)B=Bg(B^*B)$. Therefore the $2,2$ entry of the matrix above equals
$$
B^*\left( (1+BB^*)^{1/2}\left((1+BB^*)^{1/2}+1\right)\right)^{-1}-1
$$
$$
=B^*B\left( (1+B^*B)^{1/2}\left((1+B^*B)^{1/2}+1\right)\right)^{-1}-1=-(1+B^*B)^{-1/2},
$$
after elementary computations. That is
\begin{equation}\label{uno}
2m(E')-1=\left( \begin{array}{cc} (1+BB^*)^{-1/2} & (1+BB^*)^{-1/2}B \\ B^*(1+BB^*)^{-1/2} & -(1+B^*B)^{-1/2} \end{array}\right).
\end{equation}

On the other hand, in $\h_3\oplus\h_4$ (using the unitary ${\bf D}$ above)
$$
2\delta_{E'}(\frac12)-1=e^{iX_{E'}}(2P_{R(E')}-1)=e^{\left(\begin{array}{cc} 0 & i \arctan(R) D \\ -i D^* \arctan(R) & 0 \end{array}\right)} (2P_{R(E')}-1)
$$
$$
={\bf D}^* e^{\left(\begin{array}{cc} 0 & i \arctan(R) D \\ -i D^* \arctan(R) & 0 \end{array}\right)} {\bf D}  (2P_{R(E')}-1).
$$
Note that ${\bf D}(2P_{R(E')}-1)=\left(\begin{array}{cc} 1 & 0 \\ 0 & -1 \end{array} \right) {\bf D}$. Then, using the first case
$$
2\delta_{E'}(\frac12)-1={\bf D}^* \left( \begin{array}{cc} (1+R^2)^{-1/2} & (1+R^2)R \\ -R(1+R^2)^{-1/2} & (1+R^2)^{-1/2} \end{array} \right)\left(\begin{array}{cc} 1 & 0 \\ 0 & -1\end{array} \right) {\bf D}
$$
\begin{equation}\label{dos}
=\left( \begin{array}{cc} (1+R^2)^{-1/2} &  (1+B^2)^{-1/2}RD \\ D^*R(1+R^2)^{-1/2} & -D^*(1+R^2)^{-1/2}D\end{array}\right).
\end{equation}
Since $B=RD$, we have 
\begin{itemize}
\item
$(1+BB^*)^{-1/2}=(1+RDD^*R)^{-1/2}=(1+R^2)^{-1/2}$;
\item
 $(1+BB^*)^{-1/2}B=(1+R^2)^{-1/2}RD$;
\item
$B^*(1+BB^*)^{-1/2}=D^*R(1+R^2)^{-1/2}$; and
\item
$-(1+B^*B)^{-1/2}=-(1+D^*R^2D)^{-1/2}=-D^*(1+R^2)^{-1/2}D$.
\end{itemize}
Therefore (\ref{uno}) equals (\ref{dos}).
\end{proof}
\begin{rem}
Recall the properties of the matched projection cited in Remark \ref{props matched}.
For instance, if $\delta_E$ is the unique minimal geodesic such that $\delta_E(0)=P_{R(E)}$ and $\delta_E(1)=P_{R(E^*)}$, then $\delta_E^\perp(t)=1-\delta_E(t)$ is the unique minimal geodesic joining $\delta_E^\perp(0)=1-P_{R(E)}=P_{N(E^*)}$ and $\delta_E^\perp(1)=1-P_{R(E^*)}=P_{N(E)}$. It follows that the midpoints $\delta_E(\frac12)$ of $P_{R(E)}$, $P_{R(E^*)}$ and $\delta_E^\perp(\frac12)$ of $P_{N(E)}$, $P_{N(E^*)}$, satisfy
$$
\|\gamma(\frac12)-E\|\le\|P-E\|\le\|\gamma^\perp(\frac12)-E\|,
$$
for all $P\in\p(\h)$.

Note also that the properties $m(E)=m(E^*)$ and $m(1-E)=1-m(E)$ are (geometrically) evident. Indeed, in the first case, $m(E)$ is the midpoint of the unique geodesic between $P_{R(E)}$ and $P_{R(E^*)}$, which is the same as the midpoint beween $P_{R(E^*)}$ and $P_{R(E)}$, i.e., $m(E^*)$. In the latter, $m(1-E)$ is the midpoint beween $P_{R(1-E)}=P_{N(E)}=P_{R(E^*)}^\perp$ and $P_{R(1-E^*)}=P_{R(E)}^\perp$. Now the geodesic between $P^\perp$ and $Q^\perp$ is $1-\delta$, where $\delta$ is the geodesic between $P$ and $Q$. The asssertion follows.

Finally, if $\Phi:\b(\h)\to\b(\k)$ is a unital $*$-homomorphism, and $\delta(t)=e^{itX}P_{R(E)}e^{-itX}$ is the  geodesic joining $\delta(0)=P_{R(E)}$ and $\delta(1)=P_{R(E^*)}$, then $$
\Phi(\delta(t))=\Phi(e^{itX}P_{R(E)}e^{-itX})=e^{it\Phi(X)}\Phi(P_{R(E)})e^{-it\Phi(X)}.
$$
Note $\Phi(\delta)$ is also a geodesic: $\Phi(X)$ is selfadjoint and $\Phi(P_{R(E)})$-codiagonal. And also, due to the formulas (\ref{ando1}) and (\ref{ando2}), 
$$
\Phi(P_{R(E)})=\Phi(E(E+E^*-1)^{-1})=\Phi(E)(\Phi(E)+\Phi(E)^*-1)^{-1}=P_{R(\Phi(E))},
$$
and similarly $\Phi(P_{R(E^*)})=P_{R(\Phi(E)^*)}$. So that
$\Phi(m(E))=\Phi(\delta(\frac12))=m(\Phi(E))$.

\end{rem}
\section{Computation of the midpoint in terms of Halmos' theory and Davis' symmetry}

P. Halmos \cite{halmos} proved that given two projections $P,Q$ in generic position in a Hilbert space $\h$, there exists a unitary isomorphism between $\h$ and a product space $\l\times\l$, and a positive operator $X$ in $\l$ with $N(X)=\{0\}$ and $\|X\|\le\pi/2$ such that, the projections are carried (via this isomorphism) to
$$
P\simeq\left(\begin{array}{cc} 1 & 0 \\ 0 & 0 \end{array} \right) \ \hbox{ and } \ Q\simeq \left(\begin{array}{cc} C^2 & CS \\ CS & S^2 \end{array} \right),
$$
where $C=\cos(X)$ and $S=\sin(X)$. In terms of this description (and modulo this spatial isomorphism) the unique geodesic between $P$ and $Q$ is easy to describe. Indeed, note that
$$
\left(\begin{array}{cc} C^2 & CS \\ CS & S^2 \end{array} \right)=\left(\begin{array}{cc} C & -S \\ S & C \end{array} \right)\left(\begin{array}{cc} 1 & 0 \\ 0 & 0 \end{array} \right)
\left(\begin{array}{cc} C & S \\ -S & C \end{array} \right),
$$
and
$$
\left(\begin{array}{cc} C & -S \\ S & C \end{array} \right)=e^{\left(\begin{array}{cc} 0  & -X \\ X & 0 \end{array} \right)}=e^{i\left(\begin{array}{cc} 0 & iX \\ -iX & 0 \end{array} \right)}.
$$
Therefore
$$
\left(\begin{array}{cc} C & S \\ -S & C \end{array} \right)=\left(\begin{array}{cc} C & -S \\ S & C \end{array} \right)^*=e^{-i \left(\begin{array}{cc} 0 & iX \\ -iX & 0 \end{array} \right)}.
$$
That is, $\left(\begin{array}{cc} 0 & iX \\ -iX & 0 \end{array} \right)$ is the exponent of the unique geodesic 
$$
\delta(t)=e^{it\left(\begin{array}{cc} 0 & iX \\ -iX & 0 \end{array} \right)} \left(\begin{array}{cc} 1& 0 \\ 0 & 0 \end{array}\right) e^{-it\left(\begin{array}{cc} 0 & iX \\ -iX & 0 \end{array} \right)}
$$
joining $\delta(0)\simeq P$ and $\delta(1)\simeq Q$.
Therefore, if one writes the curve $2\delta-1$ of symmetries , associated to the geodesic $\delta$, it takes the form
$$
2\delta(t)-1=2e^{it\left(\begin{array}{cc} 0 & iX \\ -iX & 0 \end{array} \right)}\left(\begin{array}{cc} 1& 0 \\ 0 & 0 \end{array}\right) e^{-it\left(\begin{array}{cc} 0 & iX \\ -iX & 0 \end{array} \right)}
-1
$$
$$
=e^{it\left(\begin{array}{cc} 0 & iX \\ -iX & 0 \end{array} \right)}\left( 2\left(\begin{array}{cc} 1& 0 \\ 0 & 0 \end{array}\right)-1\right)e^{-it\left(\begin{array}{cc} 0 & iX \\ -iX & 0 \end{array} \right)}=e^{2it\left(\begin{array}{cc} 0 & iX \\ -iX & 0 \end{array} \right)}\left(\begin{array}{cc} 1& 0 \\ 0 & -1 \end{array}\right).
$$

Then
$$
2\delta(\frac12)-1=e^{2i\frac12\left(\begin{array}{cc} 0 & iX \\ -iX & 0 \end{array} \right)}\left(\begin{array}{cc} 1& 0 \\ 0 & -1 \end{array}\right)=\left(\begin{array}{cc} C & -S \\ S & C \end{array}\right)\left(\begin{array}{cc} 1& 0 \\ 0 & -1 \end{array}\right)=\left(\begin{array}{cc} C & S \\ S & -C \end{array}\right).
$$
Therefore:
\begin{teo}
Let $P$, $Q$ be in generic position. Then, modulo the Halmos isomorphism, the midpoint of the geodesic bewteen $P$ and $Q$ is 
$$
\delta(\frac12)\simeq\frac12\left(\begin{array}{cc} C+1& S \\ S & 1-C \end{array}\right).
$$
\end{teo}

\bigskip

In our case, $P=P_{R(E)}$ and $Q=P_{R(E^*)}$ may not be in generic position. The relevant decomposition to describe the position of $R(E)$ and $R(E^*)$, in terms of Halmos approach (recalling that $R(E)\cap R(E^*)^\perp=\{0\}=R(E)^\perp\cap R(E^*)$) is
\begin{equation}\label{3 espacios}
\h=R(E)\cap R(E^*)\oplus R(E)^\perp\cap R(E^*)^\perp\oplus \h_0,
\end{equation}
where $\h_0$ is the generic part. Then, the exponent $X_E$, of the geodesic between $P_{R(E)}$ and $P_{R(E^*)}$, in terms of the decomposition (\ref{3 espacios}) is, modulo the Halmos isomorphism
$$
X_E\simeq 0\oplus 0 \oplus   \left(\begin{array}{cc} 0 & iX \\ -iX & 0 \end{array} \right).
$$
This is because in the first two subspaces, $P_{R(E)}$ and $P_{R(E^*)}$ act, respectively, as the identity and zero.

In \cite{davis}, Chandler Davis established the relation between decompositions of a selfadjoint contraction $A$ as a difference of two orthogonal projections $P,Q$ (i.e., $A=P-Q$), and symmetries $V$ such that $VAV=-A$. Specifically, in the case when $P$ and $Q$ satisfy that $R(P)\cap N(Q)=N(P)\cap R(Q)=\{0\}$, there is a special symmetry $V_d$, given by
\begin{equation}\label{davis}  
V_d:=D^{-1/2}(P+Q-1),
\end{equation}
where $D:=1-(P-Q)^2$. Note that $N(1-(P-Q)^2)=R(P)\cap N(Q)\oplus N(P)\cap R(Q)=\{0\}$, but $1-(P_Q)^2$ need not be invertible, and thus $D^{-1/2}$ might be unbounded. Nevertheless, $V_d$ turns out to be a symmetry, satisfying $V_dAV_d=-A$. In \cite{p-q}, we proved the relation between $V_d$ and the geodesic joining $P$ and $Q$ (which in this case is unique), namely
\begin{equation}\label{relacion}
V_d=e^{iZ}(2P-1),
\end{equation}  
where $Z$ is the (selfadjoint) exponent of the (unique) geodesic $\delta$ joining $P$ and $Q$, or through Halmos isomorphism
$$
V_d\simeq \left(\begin{array}{cc} C & -S \\ S & C \end{array} \right)\left(\begin{array}{cc} 1 & 0 \\ 0 & -1 \end{array} \right).
$$
In other words:
\begin{coro}
The midpoint $\delta(\frac12)$ of the geodesic $\delta$ joining $P$ and $Q$ is
$$
\delta(\frac12)=\frac12(V_d+1),
$$
i.e., $V_d=2\delta(\frac12)-1$.
\end{coro}
\begin{proof}
We emphasize that in our case, $P=P_{R(E)}$ and $Q=P_{R(E^*)}$, the condition $R(P)\cap N(Q)=N(P)\cap R(Q)=\{0\}$ holds.
\end{proof}

Let us write down the symmetry $V_d$ in our special case. I turns out that it has a nicer expression in terms of reflections: write $C=C_E=2E-1$ (so that $C^*=2E^*-1$). After routine computations, we obtain that
\begin{equation}\label{los Cs}
P_{R(E)}=(1+C)(C+C^*)^{-1} \ \hbox{ and }\  P_{R(E^*)}=(1+C^*)(C+C^*)^{-1}.
\end{equation}
Notice that in our case, $D=1-A^2$ is invertible (i.e., $D^{-1}$ is bounded), because $\|A\|=\|P_{R(E)}-P_{R(E^*)}\|<1$.
\begin{teo}
With the current notations, we have that for $P_{R(E)}$ and $P_{R(E^*)}$, the Davis symmetry $V_d$ is the unitary part in the polar decomposition of $C+C^*$.
\end{teo}
\begin{proof}
We compute first 
$$
D=1-A^2=(1-A)(1+A)
$$
$$
=\{1-(1+C)(C+C^*)^{-1}+(1+C^*)(C+C^*)^{-1}\}\{1+(1+C)(C+C^*)^{-1}-(1+C^*)(C+C^*)^{-1}\}
$$
$$
=4C^*(C+C^*)^{-1}C(C+C^*)^{-1}.
$$
Then 
$$
D^{-1}=\frac14 (C+C^*)C(C+C^*)C^*=\frac14(2+C^*C+CC^*)=\frac14(C+C^*)^2,
$$
and therefore
$$
D^{-1/2}=\frac12 \{(C+C^*)^2\}^{1/2}=|C+C^*|.
$$
On the other hand, 
$$
P_{R(E)}+P_{R(E^*)}-1=(1+C)(C+C^*)^{-1}+(1+C^*)(C+C^*)^{-1}-1=2(C+C^*)^{-1}.
$$
Thus,
$$
V_d=|C+C^*|(C+C^*)^{-1}, \ \hbox{ i.e., } \ C+C^*=V_d|C+C^*|
$$
is the (unique) polar decomposition of the invertible element $C+C^*$.
\end{proof}

\section{An example}\label{seccion 7}

Let us consider the following example:
\begin{ejem}\label{ejemplo 54}
Let $\h=L^2=L^2(\mathbb{T})$ (with normalized Lebesgue measure), and for $a\in\mathbb{D}$, consider the map $\varphi_a(z)=\displaystyle{\frac{a-z}{1-\bar{a}z}}$. This map satisfies $\varphi_a(\mathbb{D})=\mathbb{D}$ and $|\varphi_a(z)|=1$, so that it also is a  map from $\mathbb{T}$ onto $\mathbb{T}$, and is its own inverse: $\varphi_a(\varphi_a(z))=z$. Therefore it induces a reflection in $L^2$,
$$
\Gamma_af =f\circ\varphi_a,\  f\in L^2.
$$
Clearly $\Gamma_a^2=1$. The idempotent $E_a=\frac12(\Gamma_a+1)$ associated to $\Gamma_a$ has range $R(E_a)=N(\Gamma_a-1)$, with $R(E_a^*)=N(\Gamma_a+1)^\perp$. These subspaces were studied in \cite{reflexiones en L2}. There it was shown that $N(\Gamma_a-1)$ and  $N(\Gamma_a+1)$ are in generic position (Theorem 6.3). Also it was shown (Theorem 3.2) that $N(\Gamma_a-1)=\Gamma_{\omega_a}(\e)$, where $\omega_a=\frac{1}{\bar{a}}(1-\sqrt{1-|a|^2})$ is the (unique) fixed point of $\varphi_a$ inside $\mathbb{D}$, and $\e$ is the subspace of elements in $\h$ which have only even powers in their Fourier series:
$$
\e=\{f\in L^2: f=\sum_{k=-\infty}^\infty \hat{f}(2k) z^{2k}\},
$$
so that
$$
R(E_a)=\{g\in L^2: g=\sum_{k=-\infty}^\infty a_k \left(\frac{\omega_a-z}{1-\bar{\omega}_az}\right)^{2k}\}.
$$
Accordingly, $N(\Gamma_a+1)=\Gamma_{\omega_a}\o$, for $\o$  the subspace of  elements which have only non nil Fourier coefficients for odd indices, and then
$$
R(E_a^*)=\{h\in L^2: h=\sum_{k=-\infty}^\infty a_k \left(\frac{\omega_a-z}{1-\bar{\omega}_az}\right)^{2k+1}\}^\perp.
$$
Let us compute the polar decomposition of $\Gamma_a+\Gamma_a^*$. It is easy to see that $\Gamma_a^*=M_{|\psi_a|^2}\Gamma_a=\Gamma_a M_{1/|\psi_a|^2}$, where $\psi_a(z)=\frac{(1-|a|^2)^{1/2}}{1-\bar{a}z}$ is the normalized  Szego kernel.
Then 
$$
|\Gamma_a+\Gamma_a^*|^2=(\Gamma_a+\Gamma_a^*)^2=(\Gamma_a+M_{|\psi_a|^2}\Gamma_a)(\Gamma_a+\Gamma_a M_{1/|\psi_a|^2})=2+M_{|\psi_a|^2}+M_{1/|\psi_a|^2},
$$
and then $|\Gamma_a+\Gamma_a^*|=M_{\gamma_a}$, where 
$$
\gamma_a=(2+|\psi_a|^2+1/|\psi_a|^2)^{1/2}=\frac{1+|\psi_a|^2}{|\psi_a|}.
$$
Then 
$$
V_d=|\Gamma_a+\Gamma_a^*|(\Gamma_a+\Gamma_a^*)^{-1}=M_{\gamma_a}(\Gamma_a+\Gamma_aM_{1/|\psi_a|^2})^{-1}=M_{\gamma_a}\{\Gamma_a(1+M_{1/|\psi_a|^2}\Gamma_a)\}^{-1}
$$
$$
=M_{\gamma_a(1+1/|\psi_a|^2)^{-1}}\Gamma_a=M_{|\psi_a|}\Gamma_a.
$$
Then. the midpoint $\delta(\frac12)$ of the geodesic between $P_{R(E_a)}$ and $P_{R(E_a^*)}$ is 
$$
\delta(\frac12)=\frac12(M_{|\psi_a|}\Gamma_a+1).
$$
Using the formula in \cite{ando} for $P_{R(E)}$ ((\ref{ando1})  above), for the idempotent $\frac12(\Gamma_a+1)$ that projects onto $N(\Gamma_a-1)$ and  for $1-\frac12(\Gamma_a+1)$ that projects onto $N(\Gamma_a+1)$, it is easy to see that
\begin{equation}
P_{N(\Gamma_a-1)}=(1+\Gamma_a)M_{(1+|\psi_a|^2)^{-1}} \ \hbox{ and } \ P_{N(\Gamma_a+1)}=(1-\Gamma_a)M_{(1+|\psi_a|^2)^{-1}}
\end{equation}
Then
$$
P_{N(\Gamma_a-1)}-P_{N(\Gamma_a+1)}^\perp=P_{N(\Gamma_a-1)}+P_{N(\Gamma_a+1)}-1=M_{\kappa_a},
$$
where $\kappa_a:=2(1+|\psi_a|^2)^{-1}-1=\frac{1-|\psi_a|^2}{1+|\psi_a|^2}$. An elementary calculation shows that the image $\{\kappa_a(z): z\in\mathbb{T}\}$ of this map is $[-|a|,|a|]$. Therefore, in particular,
\begin{equation}\label{norma ejemplo}
\|P_{N(\Gamma_a-1)}-P_{N(\Gamma_a+1)}^\perp\|=|a|.
\end{equation}
The operator $B_a:N(\Gamma_a-1)^\perp\to N(\Gamma_a-1)$ in the matrix representation $E_a=\frac12(\Gamma_a+1)=\left(\begin{array}{cc} 1 & B_a \\ 0 & 0 \end{array}\right) $ in terms of $\h=N(\Gamma_a-1)\oplus N(\Gamma_a-1)^\perp$ has trivial nullspace and dense range (this is a consequence of the fact the eigenspaces of $\Gamma_a$ are in generic position, see Section 2). Also $B_a$  can be explicitely computed: 
$$
B_a=\frac12 P_{N(\Gamma_a-1)}(1+\Gamma_a)(1-P_{N(\Gamma_a)})=\frac12M_{\kappa_a}(\Gamma_a-1).
$$
From Theorem \ref{teo 24} and (\ref{norma ejemplo}) above, it follows that
$$
\|B_a\|=\frac{|a|}{\sqrt{1-|a|^2}}.
$$
Consider the polar decomposition $B_a=U_a|B_a|$, with  $U_a:N(\Gamma_a-1)^\perp\to N(\Gamma_a-1)$    isometric. Then, with the similar  computations as in Lemma \ref{lema B} and Corollary \ref{coro 32}, we have that the exponent of the geodesic $\gamma$ joining $\gamma(0)=P_{R(E_a)}=P_{N(\Gamma_a-1)}$ and $\gamma(1)=P_{R(E_a^*)}=P_{N(\Gamma_1+1)}^\perp$ is (in  terms of the decomposition $\h=N(\Gamma_a-1)\oplus N(\Gamma_a-1)^\perp$)
$$
X_a=\left( \begin{array}{cc} 0 & i U_a \arctan(|B_a|) \\ -i \arctan(|B_a|)U_a^* & 0 \end{array}\right).
$$
Let us compute $|B_a|$ and $U_a$ explicitely.
Recall that 
$$
(E_a+E_a^*-1)^2-1=\left(\begin{array}{cc} B_aB_a^* & 0 \\ 0  & B_a^*B_a\end{array} \right)
$$
On the other hand
$$
(E_a+E_a^*-1)^2-1=E_aE_a^*+E_a^*E_a-E_a-E_a^*
$$
$$
=\frac14(1+\Gamma_a)(1+\Gamma_a^*)+\frac14(1+\Gamma_a^*)(1+\Gamma_a)-\frac12(1+\Gamma_a) -\frac12(1+\Gamma_a^*)=\frac14\Gamma_a\Gamma_a^*+\frac14\Gamma_a^*\Gamma_a-\frac12
$$
$$
=\frac14 M_{|\psi_a|^2}+\frac14M_{\frac{1}{|\psi_a|^2}}-\frac12=M_{\frac{(|\psi_a|^2-1)^2}{4|\psi_a|^2}}.
$$
Then
$$
\left(\begin{array}{cc} |B_a^*| & 0 \\ 0 & |B_a| \end{array} \right)=\frac12 M_{\frac{||\psi_a|^2-1|}{|\psi_a|}}.
$$
It follows that $|B_a|$ is the compression of $\frac12 M_{\frac{||\psi_a|^2-1|}{|\psi_a|}}$ with $P_{N(\Gamma_a-1)}^\perp$. Note that 
$$
P_{N(\Gamma_a-1)}^\perp=1-P_{N(\Gamma_a-1)}=1-(1+\Gamma_a)M_{(1+|\psi_a|^2)^{-1}}=M_{1-(1+|\psi_a|^2)^{-1}}-\Gamma_aM_{(1+|\psi_a|^2)^{-1}}.
$$
Clearly  $\Gamma_aM_{(|\psi_a|^2+1)^{-1}}=M_{\frac{|\psi_a|^2}{1+|\psi_a|^2}}\Gamma_a$. Then
$$
P_{N(\Gamma_a-1)}^\perp=M_{\frac{|\psi_a|^2}{1+|\psi_a|^2}}(1-\Gamma_a).
$$
Thus
$$
|B_a|=\frac12M_{\frac{|\psi_a|^2}{1+|\psi_a|^2}}(1-\Gamma_a)M_{\frac{||\psi_a|^2-1|}{|\psi_a|^2}}M_{\frac{|\psi_a|^2}{1+|\psi_a|^2}}(1-\Gamma_a),
$$
which after straightforwad computations (involving $\Gamma_a M_{f(|\psi_a|^2)}=M_{f(1/|\psi_a|^2)}\Gamma_a$, for $f$ continuous in $(0,+\infty)$), give
\begin{equation}
|B_a|=M_{|\psi_a|^2\frac{||\psi_a|^2-1|}{(1+|\psi_a|^2)^2}}(1-\Gamma_a).
\end{equation}
To complete the computation of the exponent $X_a$, we must describe the  isometry
$U_a:N(\Gamma_a-1)^\perp\to N(\Gamma_a-1)$. That is
$$
\frac12 M_{\kappa_a}(\Gamma_a-1)=U_a M_{|\psi_a|^2\frac{||\psi_a|^2-1|}{(1+|\psi_a|^2)^2}}(1-\Gamma_a) \ \hbox{ restricted to } N(\Gamma_a-1)^\perp.
$$
Since $\Gamma_a-1$ is an isomorphism when restricted to $N(\Gamma_a-1)^\perp$, this amounts to
$$
\frac12 M_{\kappa_a}=U_a M_{|\psi_a|^2\frac{||\psi_a|^2-1|}{(1+|\psi_a|^2)^2}} \ \hbox{ on the range of } \Gamma_a-1\big|_{N(\Gamma_a-1)^\perp},  \hbox{ i.e., on } N(\Gamma_a+1).
$$
This implies that the (non invertible) operator $M_{|\psi_a|^2\frac{||\psi_a|^2-1|}{(1+|\psi_a|^2)^2}}$ maps $N(\Gamma_a+1)$ onto a dense subspace of $\s\subset N(\Gamma_a-1)=\overline{R(B_a^*)}$. It follows that on this subspace $\s$, we have
$$
U_a\Big|_\s=M_{\kappa_a}M_{\frac{(1+|\psi_a|^2)^2}{|1-|\psi_a|^2|}}\Big|_\s=M_{sgn(1-|\psi_a|^2)} M_{\frac12(1+|\psi_a|^2)}\Big|_\s,
$$
where $sgn(t)$ is the sign function: $sign(t)=\left\{ \begin{array}{l} t/|t| \hbox{ if } t\ne 0 \\ 0 \hbox{ if } t=0 \end{array} \right.$.
Since the right hand operator is bounded, it follows that
$$
U_a=M_{sgn(1-|\psi_a|^2)} M_{\frac12(1+|\psi_a|^2)}\Big|_{N(\Gamma_a-1)}.
$$
Note that $M_{sgn(1-|\psi_a|^2)}$ is a global symmetry: the set $\{t\in[-\pi,\pi]: |\psi_a|=1\}$ has measure zero: it consists of two points (for instance, for $a=r\in(0,1)$, it is $\{\arccos(r), -\arccos(r)\}$). This implies that $M_{\frac12(1+|\psi_a|^2)}\Big|_{N(\Gamma_a-1)}$ is isometric.
\end{ejem}

\section{A retraction from idempotents onto  projections}

In \cite{pr} it was noted that the unitary part $R$ in the polar decomposition $C=R|C|$ of a reflection $C$, is a symmetry. In \cite{cpr}, G. Corach, H. Porta and L. Recht studied the geometry of the map from the set $\q(\h)=\{C\in\b(\h): C^2=1\}$ of reflections onto the set $\p(\h)=\{S\in\q(\h): S^*=S\}$ of symmetries:
\begin{equation}\label{mapa cpr}
\pi:\q(\h)\to\p(\h), \ \pi(C)=R=C|C|^{-1}.
\end{equation}
Notice that in  Example \ref{ejemplo 54},  the symmetry  $V_d=M_{|\psi_a|}\Gamma_a$ also coincides with the unitary part in the polar decomposition of $\Gamma_a$:
$$
\Gamma_a^*\Gamma_a=M_{\frac{(1-|a|^2}{|1-\bar{a}z|^2}}\Gamma_a\Gamma_a=M_{|\psi_a|^2},
$$
and thus $|\Gamma_a|=\left(M_{|\psi_a|^2}\right)^{1/2}=M_{|\psi_a|}$. Therefore the unitary part $R_a$ of the polar decomposition $\Gamma_a=R_a|\Gamma_a|$ is 
$R_a=|\Gamma_a|\Gamma_a=M_{|\psi_a|}\Gamma_a=V_d$.
It is fair to ask if this is always the case.
\begin{prop}
Let $C$ be a reflection, then the unitary part in the polar decomposition of $C+C^*$ coincides with the unitary part of the polar decomposition of $C$. 
\end{prop}
\begin{proof}
Let $C=R_C|C|$ and $C+C^*=R_{C+C^*}|C+C^*|$ be the polar decompositions of $C$ and $C+C^*$. We must show that
$$
C|C|^{-1}=(C+C^*)|C+C^*|^{-1}.
$$
This is equivalent to the equality of the inverses $|C|C=|C+C^*|(C+C^*)^{-1}$,  i.e., 
\begin{equation}\label{paso}
|C|C(C+C^*)=|C+C^*|.
\end{equation}
Both terms in (\ref{paso}) are positive: 
$$
|C|C(C+C^*)=|C|(1+CC^*)=(CC^*)^{-1/2}(1+CC^*)
$$
is positive. Then it suffices to prove the equality of the squares in (\ref{paso}): 
$$
\left(|C|(1+CC^*)\right)^2=|C+C^*|^2.
$$
The left hand term equals (again, since $|C|$ and $1+CC^*$ commute)
$$
C^*C(1+CC^*)^2=C^*C(1+2CC^*+(CC^*)^2)=C^*C+2+CC^*,
$$
which coincides with 
$$
|C+C^*|^2=(C+C^*)^2=2+CC^*+C^*C.
$$
\end{proof}
We have then yet another characterization of the midpoint between $R(E)$ and $R(E^*)$:  
\begin{coro}
The midpoint $\delta_E(\frac12)$ between $P_{R(E)}$ and $P_{R(E^*)}$ is the unitary part $R_E$ in the polar decomposition $2P_{R(E)}-1=R_E|2P_{R(E)}-1|$.

In other words, 
$$
\frac12(\pi(C)+1)=m(\frac12(C+1)).
$$
\end{coro}

\medskip

In \cite{cpr}, a Finsler metric was introduced in $\q(\h)$: for $C\in\q(\h)$ and $X\in(T\q(\h))_C$, 
\begin{equation}\label{finsler}
|X|_C:=\||C|^{1/2}X|C|^{-1/2}\|.
\end{equation}
Note that in $\p(\h)\subset\q(\h)$, this is the usual norm of $\b(\h)$. Among the facts proved in \cite{cpr}, it was shown that $(T\pi)_Q$ is contractive at every point $Q\in\q(\h)$. Also, that the metric behaves as a non positively curved metric when restricted to the fibers $\pi^{-1}(S)$, for $S$ in $\p$. In particular, any two points in the fiber $\pi^{-1}(S)$ can be joined by a unique geodesic, which is minimal along its path. The geodesic $C(t)$ joining $C(1)=C$ with $C(0)=S=\pi(C)$ inside $\pi^{-1}(S)$ is 
\begin{equation}\label{geodesica fibra}
C(t)=S|C|^t=|C|^{-t/2}S|C|^{t/2}=|C|^{-t}S, \ t\in\mathbb{R}.
\end{equation}

\begin{rem}
Note that the following maps in $\q(\h)$ are isometric for the Finsler structure of \cite{cpr}:
\begin{enumerate}\label{isometrias de Q}
\item
The map $E\mapsto 1-E$ between idempotents, at the reflection level is $C=2E-1\mapsto 2(1-E)-1=-C$, and is clearly isometric.
\item
The adjoint map $C\mapsto C^*$ is also isometric. Indeed, for $X\in(T\q(\h))_C$ (using that $|C^*|=|C|^{-1}$)
$$
|X^*|_{C^*}=\| |C^*|^{1/2}X^*|C^*|^{-1/2}\|=\| |C|^{-1/2}X^*|C|^{1/2}\|=\|(|C|^{1/2}X|C|^{-1/2})^*\|
$$
$$
=\||C|^{1/2}X|C|^{-1/2}\|=|X|_C.
$$
\end{enumerate}
\end{rem}
We may combine these facts, with the norm inequalities proved by  X. Tian, Q. Xu and C. Fu in \cite{TianXuFu}, to obtain:
\begin{coro}
Let $E$ be an idempotent (with infinite rank and nullity). Then, for any orthogonal projection $P$ and any $t\in\mathbb{R}$ we have that
$$
\|m(E)-P\|\le \||2E-1|^{-t/2}m(E)|2E-1|^{t/2}-P\|.
$$
\end{coro}
\begin{proof}
Let $C=2E-1$ and $C(t)=|C|^{-t/2}\pi(C)|C|^{t/2}$ the geodesic joining $C$ with $\pi(C)=2m(E)-1$. Then using Remark \ref{props matched}.1, for the idempotent $\frac12 C(t)$
$$
\|m(\frac12(C(t)+1))-P\|\le \|\frac12(C(t)+1)-P\|,
$$
for all $t\in\mathbb{R}$. Note that 
$$
\frac12(C(t)+1)=\frac12(|C|^{-t/2}\pi(C)|C|^{t/2}+1)=|C|^{-t/2}\frac12 (\pi(C)+1)|C|^{t/2}=|C|^{-t/2}E|C|^{t/2}.
$$
Also note that all elements in the fiber $\pi^{-1}(S)$ project onto $S$: $\pi(C(t))=S$, i.e., $m(\frac12(C(t)+1)=\frac12(S+1)$. Then the norm inequality above reads
$$
\|\frac12(S+1)-P\|\le \|\frac12(C(t)+1)-P\|,
$$
which proves our claim.
\end{proof}
\begin{rem}
With the same argument, it can be shown that for any idempotent $F$ such that $2F-1$ belongs to the fiber $\pi^{-1}(S)$ of $S=2E-1$, one has that for any projection $P\in\p(\h)$,
$$
\|m(E)-P\|\le \|F-P\|.
$$
\end{rem}
\begin{rem}
Also the property that  the tangent map $(T\pi)$ of $\pi$ is contractive can be used. Denote by $d$ the Finsler metric in $\q(\h)$: 
$$
d(E,F)=\inf \{\ell(\gamma): \gamma \hbox{ is smooth and joins } E \hbox{ and } F \hbox{ in } \q\},
$$
where, for $\gamma$ joining $\gamma(a)=E$ and $\gamma(b)=F$ 
$$
\ell(\gamma)=\int_a^b |\dot{\gamma}(t)|_{\gamma(t)} dt.
$$
Therefore, for  $E,F\in\q$ we have that
$$
d(\pi(E),\pi(F))\le d(E,F).
$$
For any curve $\gamma$ in $\q$ joining $\gamma(a)=E$ and $\gamma(b)=F$, we have that $\pi(\gamma)$ is smooth and joins $\pi(E)$ and $\pi(F)$, and since $(T\pi)$ is contractive, 
$$
|\dot{\pi(\gamma)}(t)|_{\pi(\gamma)(t)}=|(T\pi)_{\gamma(t)}\dot{\gamma}(t)|_{\pi(\gamma)(t)}\le |\dot{\gamma}(t)|_{\gamma(t)}.
$$
It follows that $\ell(\pi(\gamma))\le\ell(\gamma)$, and the assertion follows.
\end{rem}

\bigskip

In particular, for any orhogonal projection $P$ (in the same connected component as $\pi(E)$, i.e., with the same rank and nullity as $E$), one has that
\begin{equation}\label{desigualdad P E}
d(P,m(E))\le d(P,E).
\end{equation}
We may describe the situation with the following figure:
	\begin{center}
	\includegraphics[width=0.5\textwidth]{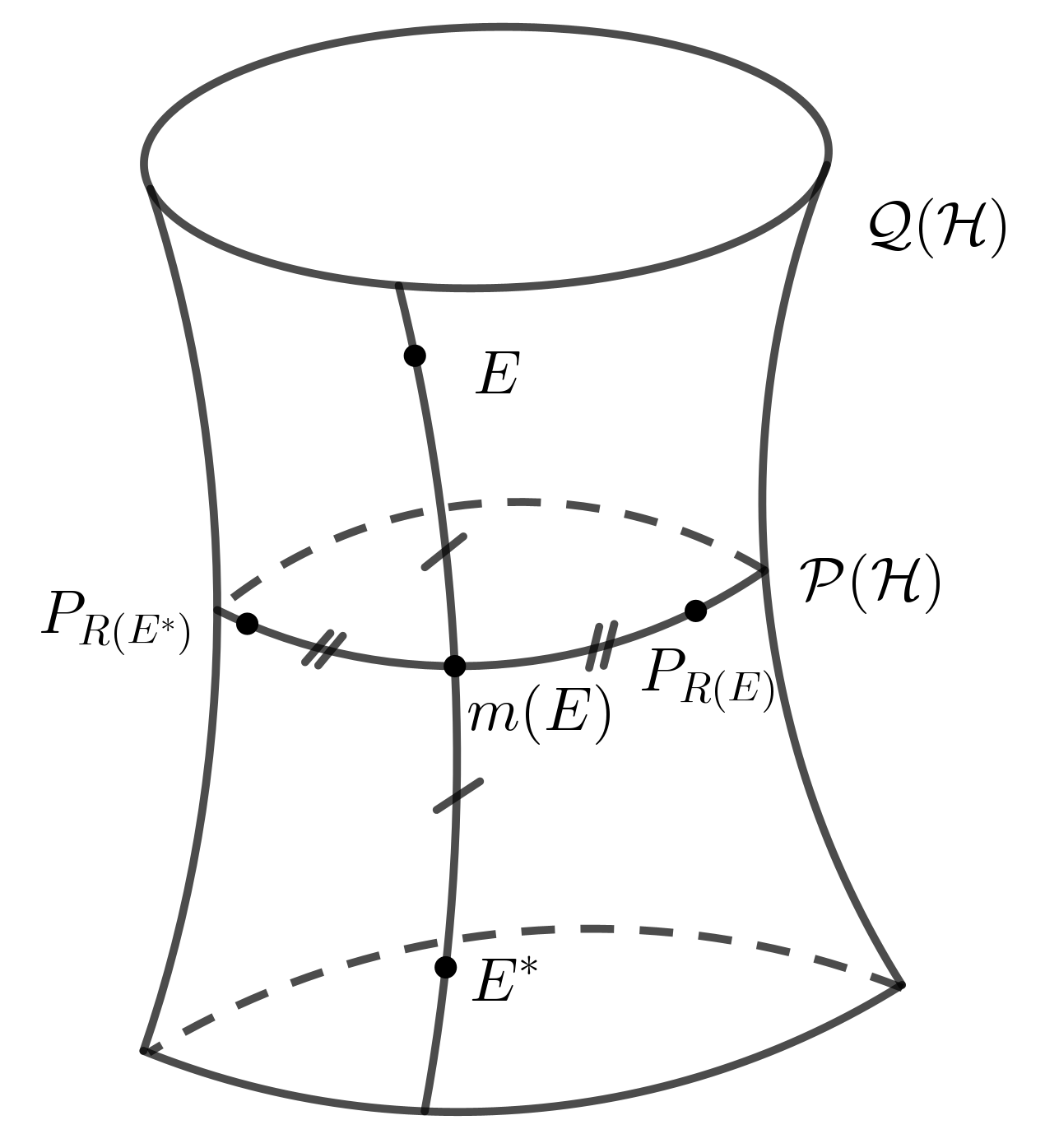}
	\end{center}
\centerline{ Figure 1}
\bigskip

Regarding this figure we have that
\begin{prop}
Identifying  an idepotent $E$   with its associated reflection $2E-1$, and a projection $P$ with the symmetry $2P-1$ .
\begin{enumerate}
\item
$$
d(E,m(E))=d(E^*,m(E))=\|\log |2E-1|\|
$$
and
$$
\|E-m(E)\|=\|E^*-m(E)\|=\frac12\left(\sqrt{\|B\|^2-1}+\|B\|\right).
$$ 
\item
$$
d(P_{R(E)},m(E))=d(P_{R(E^*)}, m(E))=\frac12\arctan(\|B\|)
$$
and
$$
\|P_{R(E)}-m(E)\|=\|P_{R(E^*)}-m(E)\|=\sin(\frac12\arctan(\|B\|)).
$$
\end{enumerate}
\end{prop}
\begin{proof}
The first part of assertion 1 is from \cite{cpr}: the unique minimal geodesic $\gamma$ joining $\gamma(1)=C=2E-1$ and $\gamma(0)=\pi(C)=2 m(E)-1$ is $\gamma(t)=\pi(C)|C|^t$, and thus $|\dot{\gamma}(t)|_{\gamma(t)}=\|\log |C|\|$.  The second assertion is Theorem 3.17 of \cite{TianXuFu}: there it was shown that $\|m(E)-E\|=\frac12\left(\|E\|-1+\sqrt{\|E\|^2-1}\right)$, and clearly $\|E\|=\sqrt{\|B\|^2+1}$.

For the second assertion,  note that $d(P_{R(E)},m(E))=\frac12\|X_E\|$, and $\|X_E\|=\arctan \|B\|$.
\end{proof}

\begin{rem}
We may compute explicitely $d(E_a,m(E_a))=\|\log(\Gamma_a)\|$ for $\Gamma_a$ of Example \ref{ejemplo 54} in Section \ref{seccion 7}. Since $\log(|C|)=\frac12\log(|C|^2)=-\frac12\log(|C^*|^{-2})$ and $|\Gamma_a^*|^2=\Gamma_a\Gamma_a^*=M_{1/|\psi_a|^2}=M_{\frac{|1-\bar{a}z|^2}{1-|a|^2}}$, we have that
$$
\log(|\Gamma_a^*|^2)=M_{\log(\frac{|1-\bar{a}z|^2}{1-|a|^2})}.
$$
the function $\log(\frac{|1-\bar{a}z|^2}{1-|a|^2})$ takes values between $\log(1-|a|^2)-\log(1+|a|^2)$ and $\log(1+|a|^2)-\log(1-|a|^2)$. It follows that 
$$
d(E_a,m(E_a))=\frac12 \log\left(\frac{1+|a|}{1-|a|}\right).
$$
Notice that the (unique, minimal) geodesic $\Gamma(t)$  of $\q(\h)$ which satisfies $\Gamma(0)=\pi(\Gamma_a)$ and $\Gamma(1)=\Gamma_a$ is (see \cite{cpr}) $\Gamma(t)=\pi(\Gamma_a)|\Gamma_a|^t$. Then (since $\pi(\Gamma_a)$ is selfadjoint)
$$
\Gamma(-1)=\pi(\Gamma_a)|\Gamma_a|^{-1}=|\Gamma_a|\pi(\Gamma_a)=\left(\pi(\Gamma_a)|\Gamma_a|\right)^*=\Gamma_a^*.
$$
That is, $\Gamma_a^*, \pi(\Gamma_a)$ and $\Gamma_a$ (in that order), belong to the same geodesic in the fiber of $\pi$. In particular, this implies that
$$
d(\Gamma_a,\Gamma_a^*)=\log(\frac{1+|a|}{1-|a|}).
$$
\end{rem}
The fact that $C^*, \pi(C)$ and $C$ (or $E^*, m(E), E$) belong to the same geodesic, at times $t=-1, t=0$ and $t=1$,  holds in general, with the same proof:
\begin{coro}
Let $C\in\q(\h)$. The geodesic $C(t)$, $t\in\mathbb{R}$ with $C(0)=\pi(C)$ and $C(1)=C$, satisfies $C(-1)=C^*$. In other words, $C$ and $C^*$ belong to the same fiber of $\pi$, and  $\pi(C)=2 m(E)-1$ is the midpoint of the unique geodesic of this fiber joining $C$ and $C^*$.
\end{coro}
This valids the scheme at Figure 1: we already have proved that $P_{R(E)}, m(E)$ and $P_{R(E^*)}$ belong o the same (unique. minimal) geodesic of $\p(\h)\subset\q$, at times $t=0, t=\frac12$ and $t=1$.

\end{document}